\documentclass{amsart}
\usepackage{amssymb}
\usepackage{braket}
\usepackage{mathrsfs}
\usepackage{ifthen}
\usepackage{graphicx}

\usepackage{tikz}
\usetikzlibrary{patterns}
\usepackage{comment} 
\usepackage{mleftright}
\usepackage{hyperref}
\usepackage{cleveref}
 \usepackage[all]{xy}
 \usepackage{amscd}
 \usepackage{amsrefs}
 \usepackage{color}
 \usepackage{enumitem}
\newlist{steps}{enumerate}{1}
\setlist[steps, 1]{label = Step \arabic*:}

\makeatletter
\DeclareRobustCommand\widecheck[1]{{\mathpalette\@widecheck{#1}}}
\def\@widecheck#1#2{%
   \setbox\z@\hbox{\m@th$#1#2$}%
   \setbox\tw@\hbox{\m@th$#1%
      \widehat{%
         \vrule\@width\z@\@height\ht\z@
         \vrule\@height\z@\@width\wd\z@}$}%
   \dp\tw@-\ht\z@
   \@tempdima\ht\z@ \advance\@tempdima2\ht\tw@ \divide\@tempdima\thr@@
   \setbox\tw@\hbox{%
      \raise\@tempdima\hbox{\scalebox{1}[-1]{\lower\@tempdima\box\tw@}}}%
   {\ooalign{\box\tw@ \cr \box\z@}}}
\makeatother

\theoremstyle{plain}
\newtheorem{thm}{Theorem}[section]
\crefname{thm}{Theorem}{Theorems}
\Crefname{thm}{Theorem}{Theorems}
\newtheorem{prop}[thm]{Proposition}
\crefname{prop}{Proposition}{Propositions}
\Crefname{prop}{Proposition}{Propositions}

\crefname{lem}{Lemma}{Lemmas}
\Crefname{lem}{Lemma}{Lemmas}
\newtheorem{cor}[thm]{Corollary}
\crefname{cor}{Corollary}{Corollaries}
\Crefname{cor}{Corollary}{Corollaries}

\crefname{claim}{Claim}{Claims}
\Crefname{claim}{Claim}{Claims}

\crefname{property}{Property}{Properties}
\Crefname{property}{Property}{Properties}

\crefname{problem}{Problem}{Problems}
\Crefname{problem}{Problem}{Problems}

\crefname{ques}{Question}{Questions}
\Crefname{ques}{Question}{Questions}

\theoremstyle{definition}

\crefname{defn}{Definition}{Definitions}
\Crefname{defn}{Definition}{Definitions}
\crefname{conj}{Conjecture}{Conjectures}
\Crefname{conj}{Conjecture}{Conjectures}

\crefname{notation}{Notation}{Notations}
\Crefname{notation}{Notation}{Notations}

\crefname{convention}{Convention}{Conventions}
\Crefname{convention}{Convention}{Conventions}

\crefname{cond}{Condition}{Conditions}
\Crefname{cond}{Condition}{Conditions}

\crefname{assum}{Assumption}{Assumptions}
\Crefname{assum}{Assumption}{Assumptions}

\crefname{conj}{Conjecture}{Conjectures}
\Crefname{conj}{Conjecture}{Conjectures}

\theoremstyle{remark}
\newtheorem{rem}[thm]{Remark}
\crefname{rem}{Remark}{Remarks}
\Crefname{rem}{Remark}{Remarks}

\crefname{ex}{Example}{Examples}
\Crefname{ex}{Example}{Examples}

\crefname{section}{Section}{Sections}
\Crefname{section}{Section}{Sections}
\crefname{subsection}{Subsection}{Subsections}
\Crefname{subsection}{Subsection}{Subsections}
\crefname{figure}{Figure}{Figures}
\Crefname{figure}{Figure}{Figures}

\newcommand{\Q}{\mathbb{Q}}

\newcommand{\R}{\mathbb R}

\newcommand{\ctext}[1]{\raise0.2ex\hbox{\textcircled{\scriptsize{#1}}}}

\def\ker{\operatorname{Ker}}

\newcommand{\mbar}[1]{{\ooalign{\hfil#1\hfil\crcr\raise.167ex\hbox{--}}}}

\title[BPS type inequalities for $s^\#$]{Bennequin--Plamenevskaya--Shumakovitch Type Inequalities for Kronheimer--Mrowka's Concordance Invariant}

\author{Nobuo Iida}
\address{Tokyo Institute of Technology, Ookayama, Meguro-ku, Tokyo}
\email{iida.n.ad@m.titech.ac.jp}

\begin{document}

\begin{abstract}
We give Bennequin--Plamenevskaya--Shumakovitch type lower bounds for the concordance invariant $s^\#$ introduced by
Kronheimer and Mrowka.
The proof is a consequence of computations for torus knots and the cobordism inequality of $s^\#$ due to Gong, combined with well-known arguments used for slice-torus invariants.
\end{abstract}
\maketitle
\tableofcontents

\section{Introduction}
As a gauge theory analogue of Rasmussen's invariant $s$ from Khovanov homology \cite{Ras}, Kronheimer and Mrowka introduced a concordance invariant $s^\#$ from singular instanton Floer theory \cite{KMRas}.
The two invariants $s$ and $s^\#$ have similarities and differences.
\par
Both of these invariants give lower bounds for the 4-ball genus $g_4(K)$ for any knot $K \subset S^3$;
\begin{equation}\label{4-ball}
s(K), s^\#(K) \leq 2g_4(K).
\end{equation}
\par
Gong \cite{Gong} gave examples of knots with $s^\# \neq s$ and moreover showed that  $s^\#$ is not additive under connected sum of knots.
\par
The aim of this paper is to study similarities and differences between $s$ and $s^\#$ from the viewpoint of contact geometry. 
A contact structure on an oriented 3-manifold $Y$ is an oriented 2-plane field $\xi$ satisfying 
\[
\xi=\ker \lambda, \quad \lambda \wedge d\lambda>0
\]
for some 1-form $\lambda$ on $Y$. 
The 3-sphere $S^3$ has a standard contact structure $\xi_{std}$.
If we regard $S^3$ as the one point compactification of $\R^3$ with the standard coordinates $(x, y, z)$, this is given by $\xi_{std}=\ker (dz+x dy)$ on $\R^3$.
\par
In a contact 3-manifold, 
there are two distinguished classes of knots.
One is Legendrian knots, which are everywhere tangent to the plane field $\xi$, and the other is transverse knots, which are everywhere transverse to the plane field.
For a survey, see \cite{Etn} by Etnyre, for example.
In the standard contact 3-sphere $(S^3, \xi_{std})$, 
a transverse knot $\mathcal{T}$ has an odd-integer-valued invariants $sl(\mathcal{T})$ called the self-linking number.  A Legendrian knot $\mathcal{L}$ has two integer-valued invariants $tb(\mathcal{L})$ and $rot(\mathcal{L})$ called the Thurston--Bennequin number and the rotation number, respectively.
These invariants are called classical invariants for transverse knots and Legendrian knots.
\par
Let $\mathcal{T}, \mathcal{L} \subset (S^3, \xi_{std})$ be a transverse knot and a Legendrian knot respectively.
Plamenevskaya \cite{Pla} and Shumakovitch \cite{Shu} independently proved that
\[
sl(\mathcal{T}) \leq s(\mathcal{T}) -1
\]
and 
\[
tb(\mathcal{L})+|rot(\mathcal{L})| \leq s(\mathcal{L})-1
\]
hold.
Combined with the 4-ball genus bound for Rasmussen invariant $s$ \eqref{4-ball}, 
these recover Rudolph's celebrated slice-Bennequin inequalities \cite{Rud}
\[
sl(\mathcal{T}) \leq 2g_4(\mathcal{T})-1
\]
and
\[
tb(\mathcal{L})+|rot(\mathcal{L})| \leq 2g_4(\mathcal{L})-1.
\]
\par
As an instanoton analogue of Plamenevskaya--Shumakovitch inequalities, one can ask whether $s^\#$ has similar lower bounds in terms of classical invariants of transverse and Legendrian knots.

We will show the following in this paper.
\begin{thm}(Bennequin--Plamenevskaya--Shumakovitch type inequalities for $s^{\#}$).\label{main}
Let $\mathcal{T}$ and $\mathcal{L}$ be a transverse knot and a Legendrian knot in the standard contact 3-sphere $(S^3, \xi_{std})$, respectively.
Then we have
\[
sl(\mathcal{T}) \leq s^{\#}(\mathcal{T}) 
\]
and 
\[
tb(\mathcal{L})+|rot(\mathcal{L})| \leq s^{\#}(\mathcal{L}).
\]
\end{thm}
\begin{rem}\label{remark}
The trefoil with suitable transverse and Legendrian representation gives equalities, so these inequalities cannot be improved by adding constants.
We can see this from the computation for torus knots \cite{Gong} 
\begin{equation}\label{torus}
s^{\#}(T_{p, q})=2g_4(T_{p, q})-1
=s(T_{p, q})-1
=\overline{sl}(T_{p, q})=(p-1)(q-1)-1.
\end{equation}
Here,  $(p, q)$ is  any pair of coprime positive integers with $p, q>1$.
Recall also that for a knot $K \subset S^3$, its maximal self-linking number $\overline{sl}(K)$ is defined to be the maximum of $sl$ among all the transverse representatives of $K$. 
The value of $g_4(T_{p, q})$ was originally determined by Kronheimer and Mrowka using singular instantons \cite{Emb1, Emb2}, which is nothing but the proof the Milnor conjecture. 
\end{rem}

\section*{Acknowledgement}
The author thanks Aliakbar Daemi, Hisaaki Endo, Hayato Imori, Tsukasa Isoshima, Itsuki Sato, Kouki Sato, Christopher Scaduto, and Masaki Taniguchi for helpful comments.
He also thanks the authors of \cite{DISST} pointing out that an alternative proof of \cref{main} can be obtained as an application of the invariant $\tilde{s}$ that they developed.
The author is supported by JSPS KAKENHI Grant Number 22J00407. 
\section{Background}
As a gauge theory analogue of Rasmussen's invariant $s$ from Khovanov homology \cite{Ras}, Kronheimer and Mrowka introduced a concordance invariant $s^\#$ using singular instanton Floer theory \cite{KMRas}.
Later, $s^\#$ was extended to an invariant for links by Gong \cite{Gong}.
Only the $s^\#$ for knots is needed for our proof of \cref{main}, but we review her extension to links.
Though we will not discuss further,  we remark that 
she also refined $s^{\#}$ into $2^{l}$ invariants $\{s^\#_I\}_{I \in \{\pm 1\}^{l}}$ for a link $L$ with $l$ components.
We also remark that $s^{\#}$ is studied from the viewpoint of equivariant singular instanton Floer theory in \cite{DISST}.

\par
Let $L \subset S^3$ be a link.
Kronheimer and Mrowka \cite{yaft, unknot} constructed a $\Q[u, u^{-1}]$ module
\[
I^\#(L)=I^\#(L; \Gamma),
\]
as one variant of singular instanton Floer homologies.
Here,  $\Gamma$ is a local coefficient system on the quotient configuration space, constructed by looking at the product of holonomies of connections along each component of links. We will omit $\Gamma$ from our notation.
 The formal parameter $u$ keeps track of the holonomy.
 We denote  the torsion-free part of $I^\#(L)$ by
 \[
 I'(L)=I^\#(L)/\text{torsion}.
 \]
 
 For a knot $K$ in $S^3$, the Floer group with local coefficient $I^\#(K)$ has a mod 4 grading and its torsion-free part $I'(K)$ has one generator in each of degrees $1$ and $-1$ mod 4.
 For the unknot $U$, 
 $I^{\#}(U)$ is a free $\Q[u, u^{-1}]$ module of rank 2.
 We denote the generators in degrees $1$ and $-1$ by $u_+$ and $u_-$, respectively.
 We regard the completion of $\Q[u, u^{-1}]$ at $u=1$ as $\Q[[\lambda]]$, where
$\lambda^2=u^{-1}(u-1)^2$.
From now on, we will regard Floer groups $I^\#(L)$ and $I'(L)$ as $\Q[[\lambda]]$ modules.

 \par
 In this paper, 
 a link cobordism $\Sigma : L_0 \to L_1$ is always assumed to be oriented and embedded in $[0, 1]\times S^3$, though cobordism maps for the singular instanton Floer  theory are defined for more general immersed cobordisms in more general 4-manifolds.
Such a link cobordism induces a $\Q[[\lambda]]$ module homomorphism
\[
\psi^\#(\Sigma): I^\#(L_0)\to I^\#(L_1).
\]

We denote the induced map on the torsion free parts by the same notation
\[
\psi^\#(\Sigma): I'(L_0)\to I'(L_1).
\]
For a link $L$  with $l$ components in $S^3$, take a connected cobordism from the unknot $\Sigma: U \to L$.
Define  $m_+(\Sigma), m_-(\Sigma)$ as the largest integer for which we can write
\[
\psi^\#(\Sigma)(u_\pm)=\lambda^{m_\pm (\Sigma)}v 
\]
for an element $v \in I'(L)$.
Define
\[
s^\#_\pm(L)=
\begin{cases}
g(\Sigma)-m_\pm(\Sigma) \quad &\text{if }\Sigma \text{ has even genus }\\
g(\Sigma)-m_\mp(\Sigma)\pm 1&\text{else}.
\end{cases}
\]
By combining these, $s^{\#}(L)$ is defined by 
\[
s^{\#}(L)=s^\#_+(L)+s^\#_-(L).
\]
Gong proved that these $s^\#_+(L), s^\#_-(L), s^{\#}(L)$  do not depend on the choice of the cobordism $\Sigma$ and give rise to link invariants.
 
\par

The following cobordism inequality is proved by Gong
\cite[Theorem 1.4]{Gong}.
We state it only for the embedded cobordisms, though she proved it more generally for immersed cobordisms.
\begin{thm}(Gong's cobordism inequality).\label{cobordism}
Let $L_1, L_2 \subset S^3$ be links and $\Sigma: L_1 \to L_2$ be a link cobordism such that every component of $\Sigma$ has non-trivial boundary in $L_1$.
Then
\[
s^{\#}(L_2)-s^{\#}(L_1)\leq -\chi(\Sigma)+|L_1|-|L_2|
\]
holds.
\end{thm}
As a special case, we obtain the following crossing change inequality.
\begin{cor}(Crossing change inequality).\label{crossing change}
Let $L_-$ be a link and $L_+$ be the link obtained by changing one crossing from positive to negative.
Then 
\[
|s^{\#}(L_+)-s^{\#}(L_-)|\leq 2
\]
holds.
\end{cor}
\begin{proof}
This follows from Gong's cobordism inequality in the same way as that of \cite{Liv04}, \cite[Proposition 1.6]{LewPhD} .
Notice that crossing change operation does not change the number of components; 
\[
|L_-|=|L_+|.
\]
The crossing change provides a genus 1 cobordism $\Sigma$ from $L_+$ to $L_-$ whose Euler characteristic is of course $\chi(\Sigma)=-2$.
This gives
\[
s^{\#}(L_-)-s^{\#}(L_+)\leq 2.
\]
Notice that $\Sigma$ with the same orientation can be regarded as a cobordism between mirrors $L^\dagger_-\to L^\dagger_+$, but this gives no new information.
By reversing the orientation of $\Sigma$, we obtain 
a genus 1 cobordism $-\Sigma: L_- \to L_+$ whose Euler characteristic is of also $-2$.
This gives
\[
s^{\#}(L_+)-s^{\#}(L_-)\leq 2.
\]
These inequalities give the desired result.
\end{proof}

\section{Proof of the main theorem}
Let us prove the Bennequin--Plamenevskaya--Shumakovitch type inequalities for $s^{\#}$, \cref{main}.
The following arguments are well-known for slice-torus invariants \cite{Shu, Kaw1, Kaw2, Liv04, LewPhD, Lew, Lobb, Abe, BS} (Even though $s^\#$ is not a slice-torus invariant, the argument works).
\begin{proof}
By the standard technique of transverse push-off, the inequality for Legendrian knots follows from the one for transverse knots.
See \cite[Lemma 2.22]{Etn}, for example.
Thus, it is enough to prove the inequality for transverse knots.
We have the transverse analogue of Alexander's theorem and Markov's theorem on the correspondence of knots and braids (See \cite[Theorem 2.10 and Theorem 2.11]{Etn}). Under this correspondance, for a transverse knot $\hat{\beta}$ which is given as the closure of a braid $\beta$, we have Bennequin's formula 
\[
sl(\hat{\beta})= x_+-x_--n,
\]
 where we asummed that $\beta$ has $n$ strands, $x_+$ positive crossings, and $ x_-$ negative crossings.
The inequality for transverse knots follows from  the proposition below for braids.
\end{proof}

\begin{prop}
Let $\beta$ be a braid with $n$ strands, $x_+$ positive crossings, and $x_-$ negative crossings. Suppose its braid closure $\hat{\beta}$ is a knot. 
Then 
\[
s^{\#}(\hat{\beta})\geq x_+-x_--n
\]
holds.
\end{prop}
\begin{proof}
We follow the arguments in Livingston's paper \cite{Liv04} and Lewark's thesis \cite{LewPhD}.
First, we will prove the inequality for positive braids, in other words, for the case with $x_-=0$ .

Take $l \geq  x_+$ so that $(n, l)$=1.
As in  \cite{Liv04},  \cite[Lemma 1.5]{LewPhD}, we have a knot cobordism $\Sigma$ from $\hat{\beta}$ to
the torus knot $T_{n, l}$ with Euler characteristic $\chi(\Sigma)= x_++l-ln$.
Gong's cobordism inequality, \cref{cobordism} and her computations for torus knots \eqref{torus} give
\[
s^\#(\hat{\beta})\geq s^\#(T_{n, l})+\chi(\Sigma)
\]
\[
=(n-1)(l-1)-1+( x_++l-ln)
\]
\[
= x_+-n.
\]

This completes the proof for the case with $x_-=0$.
\par
The general case  follows from this and the crossing change inequality, \cref{crossing change}.
 Indeed, let  $\hat{\beta}^+$ be the braid obtained from $\hat{\beta}$ by changing all the negative crossing into positive crossings.
 Since $\hat{\beta}^+$ is a positive braid with $x_+ +x_-$ crossings, we have
 \[
 s^{\#}(\hat{\beta}^+)\geq x_++x_--n.
 \]
 On the other hand, by using crossing change inequality \cref{crossing change} $x_-$ times, we have
 \[
 s^{\#}(\hat{\beta}^+)- s^{\#}(\hat{\beta}) \leq 2 x_-.
 \]
 
 Thus we obtain
 \[
 s^{\#}(\hat{\beta}) \geq s^{\#}(\hat{\beta}^+)-2 x_- \geq x_+-x_--n,
 \]
 which is exactly the desired inequality.
 \end{proof}
\begin{rem}
 The authors of \cite{DISST} pointed out that as an application of the invariant $\tilde{s}$  they developed, we can give an alternative proof of \cref{main}.
 It is proved in \cite[Theorem 1.1]{DISST} that
 $2\tilde{s}$ is a slice-torus invariant in the sense of \cite{Lew}.
 By doing the same argument for $2\tilde{s}$ instead of $s^\#$ as in the proof of \cref{main}, we have
\[
sl(\mathcal{T}) \leq 2\tilde{s}(\mathcal{T})-1.
\]
Obtaining this type of inequality by such an argument for slice-torus invariants is well-known (See \cite[Theorem 6.1]{BS}\cite[Theorem 5]{Lew}, for example).
\cite[Theorem1.1]{DISST} shows
 \[
 |s^\#(K)-2\tilde{s}(K)| \leq 1.
 \]
 In particular we have
 \[
 2\tilde{s}(K) -s^\#(K)\leq 1.
 \]
 Thus we obtain 
\[
sl(\mathcal{T}) \leq 2\tilde{s}(\mathcal{T})-1 \leq s^\#(\mathcal{T}).
\]
\end{rem}

 \bibliographystyle{jplain}
\bibliography{Bennequin}
\end{document}